\documentclass[12pt,a4paper,reqno]{amsart}
\usepackage{amsmath,amssymb,amsthm,amscd,mathrsfs} 
\usepackage[T1]{fontenc}
\usepackage[latin1]{inputenc}
\usepackage[english]{babel}
\usepackage[mathscr]{eucal}
\usepackage{graphicx}
\usepackage{float}
\usepackage{tikz}
\usetikzlibrary{positioning}
\usetikzlibrary{shapes.geometric}
\usetikzlibrary{fit}
\usetikzlibrary{patterns}
\usepackage{caption}
\usepackage{tikz-cd}
\usepackage{cases}
\usepackage{hyperref}
\usepackage{chngcntr}
\usepackage{enumitem}
\usetikzlibrary{decorations.text}

\input{comment.sty}
\includecomment{FZ}
\includecomment{MM}
\counterwithin{equation}{section}

\renewcommand{\a}{\alpha}

\renewcommand{\th}{\theta}

\renewcommand{\l}{\lambda}

\newcommand{\m}{\mu}
\newcommand{\n}{\nu}

\newcommand{\ph}{\phi}
\def\ph{\phi}
\def\md#1{\ \mbox{\rm(mod }{#1})}

\def\ol{\overline}
\renewcommand{\t}{\tau}

\def\ph{\phi}

\newcommand{\Q}{{\mathbb Q}}
\newcommand{\Z}{{\mathbb Z}}

\newcommand{\F}{{\mathbb F}}

\def\md#1{\ \mbox{\rm(mod }{#1})}

\newcommand{\pF}{\mathfrak p}

\newtheorem{theorem}{Theorem}[section]

\newtheorem{corollary}[theorem]{Corollary}

\theoremstyle{definition}

\theoremstyle{remark}

\newtheorem{remark}[theorem]{Remark}

\begin{document}
	\title[]{ On indices  of quartic number fields defined by  quadrinomials}
\textcolor[rgb]{1.00,0.00,0.00}{}
\author{  Hamid Ben Yakkou}\textcolor[rgb]{1.00,0.00,0.00}{}
\address{Faculty of Sciences Dhar El Mahraz, Sidi mohamed ben Abdellah University, P.O. Box  1874 Atlas-Fez,   Morocco}\email{beyakouhamid@gmail.com}
\maketitle
\vspace{0.3cm}
\begin{abstract}  
Consider a   quartic number field  $K$ generated by  a root  of an  irreducible   quadrinomial of the form   $ F(x)= x^4+ax^3+bx+c \in \Z[x]$. Let $i(K)$ denote the index of $K$.   Engstrom  \cite{Engstrom} established that $i(K)=2^u \cdot 3^v$ with $u \le 2$ and $v \le 1$. In this paper, we provide sufficient conditions on $a$, $b$ and $c$ for $i(K)$ to be divisible by $2$ or $3$, determining the exact corresponding values of $u$ and $v$ in each case.
 In particular, when $i(K) \neq 1$,  $K$ cannot be monogenic.  We also identify new  infinite parametric families of monogenic quartic number fields generated by roots of  non-monogenic quadrinomials.   We illustrate our results by   some computational   examples. Our method is  based on a theorem of Ore on the decomposition of primes in number fields \cite{Nar,O}.
\end{abstract}
\maketitle

\section{introduction }
Let $K$ be a    number field    of degree $n$ with  a ring of integers  $\mathcal{O}_K$. The field $K$ is called monogenic if $\mathcal{O}_K=\Z[\theta]$ for some primitive element   $\theta$ of $O_K$. Equivalently,  $(1, \th, \ldots, \th^{n-1})$ forms a $\Z$-basis of $\mathcal{O}_K$. Such a  basis is called a power integral basis of $\mathcal{O}_K$.
If the  ring $\mathcal{O}_K$  does not have any  power integral basis, then $K$ is said to be non-monogenic. The monogenity of number fields and the construction of  power integral bases have been largely studied in old and recent research, and there are many interesting open problems related to these topics which are the object of intense studies  (cf. \cite{ANH6IJAC, Bhargave, R, EG, G19,BGGy6, Gyoryredecide, GyorySeminarFrensh, Gyoryrdiscriminant, JW, PethoZigler}). 
 
  In   \cite{Gyoryredecide},  Gy\H{o}ry  proved  that there are only finitely many $\Z$-equivalence classes of $\th$ which which yield  a power integral basis  of $\mathcal{O}_K$, and a full system of representations can be determined effectively. This yields to first   general algorithms for deciding whether $K$ is monogenic or not and for determining all  power integral bases  of $\mathcal{O}_K$. 
  Evertse and Gy\H{o}ry's  book \cite{EG}  provides  detailed surveys on the discriminant form and index form theory and its applications, including related Diophantine equations and monogenity of number fields.  See also \cite{BerEVGymultiply} by Bérczes,  Evertse and Gy\H{o}ry and \cite{Evertse} by Evertse.
  
Several authors have considered the problem of monogenity in  different families of  quartic number fields. We refer to   \cite{Bhargave} by   Alpöge,  Bhargava and  Shnidman,    \cite{Akhtari} by Akhtari, \cite{Nyul} by Arn\'{o}czki and  Nyul, \cite{Funakura} by Funakura, \cite{GPP41993} by Ga\'al,  Peth\H{o} and  Pohst,  \cite{OdjTogbeZigler} by Odjoumani, Togbé   and  Ziegler,  \cite{PethoZigler} by Peth\H{o},  Ziegler, and \cite{Smithquartic} by Smith.  

Using a refined version \cite{Gyory1998} of  the general approach of \cite{Gyoryredecide} and by  using Baker's method and  efficient reduction and enumeration algorithms,   Ga\'al and  Gy\H{o}ry \cite{GGy5},  Bilu, Ga\'al and Gy\H{o}ry \cite{BGGy6} described  algorithms to solve index form equations in quintic and  sextic fields respectively.   For further results and efficient algorithms for several classes of number fields, see  the books \cite{EG} by  Evertse and Gy\H{o}ry and \cite{G19} by  Ga\'al. See  also  Narkiewicz's book \cite{Na}.

 An efficient way to study the monogenity of a number field is to study its index. For any  $\th$ of  $\mathcal{O}_K$, let $\mbox{ind}(\th)$ denote the module  index of $\Z[\th]$  in $\mathcal{O}_K$. In this paper, $i(K)$ will denote the index of  $K$ defined as
\begin{equation}\label{Defi(K)}
	i(K)=\mbox{gcd}\{ \mbox{ind}(\th):  \theta \in \mathcal{O}_K \, \mbox{and} \, K=\Q(\th) \}.
\end{equation}
The prime divisors of $i(K)$ are called the common index divisors of $K$. The existence of such  primes implies the non-monogenity of $K$.  In \cite{Engstrom}, Engstrom gave  some  explicit formulas for computing  the highest power $\n_p(i(K))$  of a prime  $p$ dividing $i(K)$ following  the  factorization  type  of $p$ in $\mathcal{O}_K$ and used these results to compute $\n_p(i(K))$ for all number fields of degrees less than or equal to  $7$. Engstrom's results
were generalized by Nart \cite{Nartindex}, who developed a $p$-adic characterization of the index of a number field. The problem of determination of $\n_p(i(K))$ is stated as Problem 22 of  Narkiewicz's book  \cite{Na}, and it remains unresolved both computationally and theoretically. For  several  results regarding   indices, monogenity  and non-monogenity in certain classes of pure number fields or   number fields defined by trinomials of form  $x^n+ax^m+b$, 
we refer to  \cite{BRM,BBAMH8,BDB,BFpr,Ibaraetal,JonesASM,JW,LN}.

 In the present paper, we study  the  monogenity of  quartic  number fields defined by   irreducible quadrinomials of the form  $x^4+ax^3+bx+c$.  We give sufficient conditions so that $K$ admits a common index divisor. The presence  of such divisors leads to infinite parametric families of quartic number fields without power integral basis.   We also give  new classes  of monogenic quartic number fields generated by roots of non-monogenic quadrinomials.  The distinctive aspect of our research lies in the results we have achieved, which do not assume any constraint on the absolute discriminant  or on the  Galois group of these fields. This differs from the majority of previous studies focused on quartic number fields with relatively small discriminants and prescribed Galois groups. Furthermore, our study delves into quartic number fields defined by quadrinomials, rather than just binomials or trinomials. 
\section{Main results}
Throughout this section, $K$ is a quartic number field generated by a root $\a$ of an irreducible  polynomial   $F(x)=x^4+ax^3+bx+c \in \Z[x]$. For a prime $p$ and for any  rational integer $m$, let $\n_p(m)$ denote the $p$-adic valuation of $m$ and  $m_p := \frac{m}{p^{\n_p(m)}}$. 
 For rational integers $a_1, a_2, a_3,  b_1, b_2\, \mbox{and}\, b_3$, we mean by $(a_1, a_2, a_3)  \equiv (b_1, b_2, b_3) \md{p}$  that $a_i \equiv b_i \md{p}$ for $i=1, 2, 3$.  Also, $k$ and $l$ denote two positive integers. 
 
  With the above  notations, our main results are as follow:
\begin{theorem}\label{p=2}  Under the above notations,   we have the following statements:
	\begin{enumerate}
	\item If  $(a, b, c) \in \{ (1, 3, 7), (1, 7, 3), (3, 1, 7), (3, 5, 3), (5, 3, 3), (5, 7, 7), (7, 1, 3), (7, 5, 7)\} \md{8}$, then  $\n_2(i(K))=1$.
	\item If  $(a, b, c) \in \{(1, 5, 9), (1, 13, 1), (1, 21, 25), (1, 29, 17), (5, 1, 9), (5, 9, 1), (5, 17, 25), (5, 25, 17),  \\ (9, 5, 1),  (9, 13, 25), (9, 21, 17), (9, 29, 9), (13, 1, 1), (13, 9, 25), (13, 17, 17), (13, 25, 9),  (17, 5, 25), \\ (17, 13, 17), (17, 21, 9), (17, 29, 1), (21, 1, 25), (21, 9, 17), (21, 17, 9), (21, 25, 1),   (25, 5, 17), \\  (25, 13, 9), (25, 21, 1), (25, 29, 25), (29, 1, 17), (29, 9, 9), (29, 17, 1), (29, 25, 25), (3, 7, 5), (3, 15, 29),\\  (3, 23, 21), (3, 31, 13), (7, 3, 5), (7, 11, 29), (7, 19, 21), (7, 27, 13),   (11, 7, 29), (11, 15, 21), (11, 23, 13), \\ (11, 31, 5), (15, 3, 29), (15, 11, 21), (15, 19, 13), (15, 27, 5),  (19, 7, 21), (19, 23, 5), (19, 7, 21), \\ (19, 15, 13), (23, 3, 21), (23, 11, 13), (23, 19, 5), (23, 27, 29), (27, 7, 13), (27, 15, 5), (27, 23, 29), \\ (27, 31, 21), (31, 3, 13), (31, 11, 5), (31, 19, 29),  (31, 27, 21)\}\md{32}$, then $\n_2(i(K))=1$.	
	\item  If $(a, b, c) \in \{(0, 0, 15), (0, 8, 7), (4, 4, 7), (4, 12, 15), (8, 0, 7), (8, 8, 15), (12, 4, 15),\\  (12, 12, 7)\} \md{16}$, then   $\n_2(i(K))=1$.
\item If  $(a, b) \in \{(0, 4), (4, 8), (8, 12), (12, 0)\} \md{16}$ and $a+b \equiv -(1+c) \md{64}$, then  $\n_2(i(K))=1$.
\item If $(a, b) \in \{(2, 6), (10, 14), (18, 22), (26, 30)\} \md{32}$ and $a+b  \equiv -c+63 \md{128}$, then $\n_2(i(K))=1$.  
\item If $(a, b) \in \{(6, 42), (22, 122), (38, 74), (54, 26), (70, 106), (86, 58), (102, 10), (118, 90)\} \md{128}
$ and $a+b \equiv -(1+c) \md{1024}$, then $\n_2(i(K))=2$. 
  \item 	If $(a, b, c) \equiv (6, 10, 15) \md{16}, \mu  > 7$ is even and  $2 \n > \mu+3$, then   $\n_2(i(K))=1$.
\item 	If $(a, b, c) \equiv (6, 10, 15) \md{16},  \n>5$ and $ 2 \n < \mu+3$, then  $\n_2(i(K))=2$, where $\mu= \n_2(1+a+b+c)$ and $\n=\n_2(4+3a+b)$.
 \item  If $(a, b, c) \equiv (1, 0, 0) \md{2}, 3 \n_2(b)<2\n_2(c)$ and $\n_2(b)$ is odd, then  $\n_2(i(K))=1$.
\item  If $ b \equiv 4 \md{8}, ab_2+1 \equiv 0 \md{4}$ and $c \equiv 0 \md{32}$, then  $\n_2(i(K))=1$.
	\item If	$b \equiv 4 \md{8}, ab_2+1 \equiv 2 \md{4}$ and $c \equiv 16 \md{32}$, then  $\n_2(i(K))=1$.
 \item If $b \equiv 4 \md{8}, ab_2+1 \equiv 0 \md{8},  c \equiv 16  \md{32}$ and $b_2^4+c_2  \equiv 0 \md{4}$, then  $\n_2(i(K))=2$.
 \item  If $\n_2(b)=2k, ab_2+1 \equiv 0 \md{4}$ and $ \n_2(c)=3k+1$ with $k \ge 2$,  then $\n_2(i(K))=1$.
 \item 	If $\n_2(b)=2k,   ab_2+1 \equiv 2 \md{4}$ and $\n_2(c)=3k+l$ with $k, l \ge 2$.  In this case, $\n_2(i(K))=1$.
\item If $\n_2(b)=2k, ab_2+1 \equiv 0 \md{4},  \n_2(c)=3k+l$ and $2^{k-2}b_2^4+b_{2}^2  \cdot \frac{ab_2+1}{4}+2^{l-2}c_2 \equiv 0 \md{2}$ with $k, l \ge 2$, then $\n_2(i(K))=2$. 
 	\item If  $b \equiv 4 \md{8},  ab_2+1 \equiv 2 \md{4}, \n_2(c)=2+l$ with $l \ge 2$, $\t < 3+2 \sigma $ and  $\t$ is even, then $\n_2(i(K))=1$.
 	\item If  $b \equiv 4 \md{8},  ab_2+1 \equiv 2 \md{4}, \n_2(c)=2+l$ with $l \ge 2$ and $\t > 3+2 \sigma $, then $\n_2(i(K))=2$,  	where  $\t= \n_2\big(b_2^4+b_2^2 \cdot  \frac{ab_2+1}{2}+2^{l-1}c_2\big)$ and $\sigma = \n_2\big(2b_2^2+ \frac{3ab_2+1}{4}\big)$.
	\end{enumerate}
In particular,  $2$ is a common index divisor of $K$	and $K$ cannot be  monogenic.
\end{theorem}
The following corollary follows immediately from  the above theorem.
\begin{corollary}If any  of the following conditions holds:
\begin{enumerate}
\item $a \equiv 1 \md{8}, b \equiv 3 \md{8}$  and $c \equiv 7  \md{8}$. 
\item $a \equiv 5 \md{8}$ and $b \equiv c \equiv 3, 7 \md{8}$.
\item  $a \equiv c \equiv 7 \md{8}$ and $b \equiv 5 \md{8}$,
\end{enumerate}
then $K$ is not monogenic. 
\end{corollary}
The following result is about a class of pure quartic number fields,  and it is a direct consequence of the above theorem.
\begin{corollary}
Let $m \neq \pm 1$ be a rational integer such that $x^4-m$ is irreducible over $\Q$.   If $m \equiv 1 \md{16}$, then $ \n_2(i(K)) =1$, where $K$ is the pure quartic field $\Q(\sqrt[4]{m})$. In particular, $K$ is not monogenic. 
\end{corollary}
\begin{remark}
The above corollary is a part of  Theorem 5 in Funakura's paper  \cite{Funakura}. 
\end{remark}
Also,   when $a=0$, we deduce from Theorem \ref{p=2} the following result:
\begin{corollary}
Let $F(x)=x^4+bx+c \in \mathbb{Z}[x]$ be an irreducible polynomial and $K=\Q(\a)$ with $F(\a)=0$. If any of the following conditions holds:
\begin{enumerate}
\item $b \equiv 0 \md{16}$ and $c \equiv -1 \md{16} $.
 \item  $b \equiv 8 \md{16}$ and $c \equiv 7 \md{16}$,
\end{enumerate}
then $\n_2(i(K))=1$ and $K$ cannot be monogenic. 
\end{corollary}
\begin{remark}
It may be pointed out that the above corollary also follows from the results of \cite{DS} by  Davis and   Spearman when the authors studied the index of the quartic  number field defined by  an irreducible trinomail $x^4+bx+c$.  
\end{remark}
The subsequent theorem provides sufficient conditions on $a, b$ and $c$ so that $3$ is a common index divisor of $K$. Recall that by  \cite{Engstrom}, if $3$ divides $i(K)$, then $\n_3(i(K))=1$.
\begin{theorem}\label{p=3} With the above notations, let  $  \m= \n_3(1-a-b+c), \n= \n_3(-4+3a+b), \omega= \n_3(1+a+b+c)$ and $\t= \n_3(4+3a+b)$.  If  any one of the following conditions holds:
	\begin{enumerate}
	\item   $a \equiv \pm 1 \md{3}$, $b \equiv c \equiv 0 \md{3}$, $ \, 3  \n_3(b) < 2 \n_3(c), \, \n_3(b)$ is even and $b_3 \equiv -a \md{3}$.
	 \item $(a, b) \in \{(0, 4), (3, 22), (6, 13), (9, 4), (12, 22), (15, 13), (18, 4), (21, 22), (24, 13)\} \md{27},  a+b \equiv 1+c \md{243}$ and one of the following conditions is satisfied:
	\begin{enumerate}[label=(\roman*)]
\item  $2 \n < \mu +1 $.
\item $2 \n >  \mu +1$,  $\mu$ is  odd and $\big(1-a-b+c\big)_3 \equiv 1 \md{3}$.  
	\end{enumerate}
 \item $(a, b) \in \{(0, 23), (3, 14), (6, 5), (9, 23), (12, 14), (15, 5), (18, 23), (21, 14), (24, 5)\} \md{243},  a+b\equiv -(1+c) \md{243}$ and one of the following conditions is verified:
	\begin{enumerate}[label=(\roman*)]
		\item  $2 \t < \omega + 1 $.
		\item $2 \t >  \omega + 1$,    $\omega$ is odd and $(1+a+b+c)_3 \equiv 1 \md{3}$.  
\end{enumerate}
\item $(a, b) \in \{ (2, 16), (5, 7),  (8, 25),  (11, 16), (14, 7), (17, 25),   (20, 16), (23, 7),  (26, 25)\} \md{27}\}$ and $a+b \equiv c+1 \md{81}$.

	\end{enumerate}
then $3$ is a common index divisor of $K$. Furthermore,  $ \n_3(i(K))=1$. In particular, $K$ is  not monogenic.
\end{theorem}
As an application of the above  theorem, the following corollary gives  explicit infinite parametric families of non-monogenic quartic number fields.
\begin{corollary} With the above notations, 
if any  of the following conditions holds: 
 \begin{enumerate}
 \item  $a \equiv \pm 1 \md{3}, b \equiv 9, 18 \md{27}$ and $c \equiv 0 \md{81}$.
 \item $(a, b, c) \in \{(5, 7, 11), (5, 34, 38), (5, 61, 65), (32, 7, 38), (32, 34, 65), (32, 61, 11), (59, 7, 65),\\ (59, 34, 11), (59, 61, 38)\}\md{81}$,
 \end{enumerate}
 then  $K$ is not monogenic.  
\end{corollary}  
Recall that a monic  irreducible polynomial $F(x)  \in \Z[x]$ is called monogenic if  $\mathcal{O}_K=\Z[\a]$, where $K=\Q(\a)$ and $F(\a)=0$.  Note that the monogenity of  $F(x)$ implies that of the field $K$, but the converse is not true (see \cite{JonesASM,JW} by Jones et al.). 
 The following theorem gives  infinite families of monogenic quartic number fields defined by non-monogenic quadrinomials.
\begin{theorem}\label{monogenic} Let $F(x)=x^4+ax^3+bx+c \in \Z[x]$ with discriminant $D$.  Assume that there exists a prime $p$ such that $ \min \{ 4\n_p(a), \frac{4}{3} \n_p(b)\} > \n_p(c)  > 2$, $\n_p(c)$ is odd and $D_p$ is square free. Then $F(x)$ is irreducible over $\Q$. Let $K=\Q(\a)$,  with $\a$ a root of $F(x)$. Then $F(x)$ is not monogenic, but $K$ is monogenic.
\end{theorem}
As an application of the above theorem, the following two  corollaries identify   certain classes of monogenic  quartic number fields  defined by non-monogenic trinomials of the form $x^4+bx+c$ and $x^4+ax^3+c$ respectively: 
\begin{corollary}\label{mongentrinom1} Let  $F(x)=x^4+bx+c$. If any  of the following conditions holds:
	\begin{enumerate}
	\item $b \equiv c  \equiv 8 \md{16}$ and $32 c_2^3-27 b_2^4$ is square free.
	\item   $ b \equiv 0 \md{16}, c \equiv 8 \md{16}$ and $c_2^3-27 \cdot  2^{4(\n_2(b)-4)+1} \cdot b_2^4$ is square free.
	\item $b \equiv 0 \md{27}, c \equiv 27 \md{81}$ and $256 c_3^3-729 b_3^4$ is square free,
	\end{enumerate} 
	then $F(x)$ is irreducible  in $\Q(x)$ and  not monogenic, but a  field generated by a root of $F(x)$ is monogenic. 
\end{corollary}
\begin{corollary}\label{mongentrinom2} Let $F(x)=x^4+ax^3+c \in \Z[x]$. If any  of the following conditions holds:
\begin{enumerate}
	\item $c=8, a \equiv 2 \md{4}$ and $32-27a_2^4$ is square free.
\item $c=8, a \equiv 4 \md{8}$ and $8-27a_2^4$ is square free.
\item $c=8, a \equiv 0 \md{8}$ and $1-27 \cdot 2^{4 \n_2(a)-11} \cdot a_2^4$ is square free.
\item  $c = 2^{2k+1}, \n_2(a)=4l$ with $k \ge 2$ and $l  \ge 1$, and one of the following conditions is verified:
\begin{enumerate}[label=(\roman*)]
\item $2k+9 > 4l$ and $2^{2k-4l+9}-27a_2^4$ is square free.
\item  $2k+9 < 4l$ and $1-27 \cdot 2^{4l-2k-9} \cdot a_2^4$ is square free.
\end{enumerate}
\item $c=27, a \equiv 0 \md{3}$ and $256-81a_3^4$ is square free,
\end{enumerate}
	then $F(x)$ is irreducible  in $\Q(x)$ and not monogenic, but a  field generated by a root of $F(x)$ is monogenic. 
\end{corollary}

\section{Preliminary results}
Let  $p$ be a prime and $K$ a number field with   ring of integers $\mathcal{O}_K$.  Throughout this paper, if   $p\mathcal{O}_K=\pF_1^{e_1}\cdots\pF_g^{e_g}$ is the   factorization of $p\mathcal{O}_K$ into a product of powers of distinct  prime ideals in $\mathcal{O}_K$ with residue degrees $f(\pF_i/p)  =
 f_i, i=1, \ldots, g$, then we write $p\mathcal{O}_K = [f_1^{e_1}, \ldots, f_g^{e_g}]$.   The following table, which   is  extracted from \cite{Engstrom},   provides the value  $\n_p(i(K))$ for $ p\in \{2, 3\}$ for any quartic number field $K$. 

	\begin{table}[h!]
	\centering
	\begin{tabular} { | c | c | c |}   
		\hline
	factorization of 	$p\mathcal{O}_K$ & $\n_2(i(K))$ & $\n_3(i(K))$    \\
		\hline
	$[1^1, 1^1, 1^1, 1^1]$ & $2$  &$1$ \\
		\hline
	$[1^1, 1^1, 1^2]$ & $1$   & $0$ \\
		\hline
		$[2^1, 2^1]$ & $1$   &$0$ \\
			\hline	
	\mbox{otherwise}	& $0$   & $0$\\
		\hline	
	\end{tabular}
	\caption{ \small Prime power decomposition of the index of a  quartic  number field}
	\label{Tablequartic}
\end{table}
In order to  use Table \ref{Tablequartic}, we need to determine the  factorization type  of $2$ and $3$ in $K$. To reach this, we will use Newton polygon techniques as introduced by Ore \cite{O} and further developed by  Gu\`{a}rdia,  Montes and Nart \cite{Nar}.

 In what follows, for the convenience of the reader and  as it is necessary for the proofs of our main results, we briefly describe the use of these  techniques. For more details, we refer to \cite{Nar}.

Let    $\n_p$ denote the $p$-adic valuation of $\Q_p(x)$ and $\F_p$ stands for the finite field with $p$ elements.  Let  $F(x) \equiv  \prod\limits_{i=1}^t \phi_i(x)^{l_i} \md{p}$ be the factorization  of  $F(x)$ 
into a product of powers of distinct monic  irreducible polynomials in $\F_p[x]$. Fix  $\ph(x) = \ph_i(x)$ for some $i=1, \ldots, t$ and let  $F(x)=a_{l}(x)\phi(x)^l+a_{l-1}(x)\phi(x)^{l-1}+\cdots+a_{1}(x)\phi(x)+a_0(x)$ be the  $\phi$-adic development of $F(x)$.
 The lower convex hull  of the points $\{(j,v_p(a_j(x))): \, a_j(x)\neq0,  j=0, \ldots, l\}$ in the euclidean plane  is called the $\phi$-Newton polygon of $F(x)$ with respect to  $\n_p$,  and denoted by $\mathcal{N}_{\phi}(F)$. The polygon  $\mathcal{N}_{\phi}(F)$  is the union of different adjacent sides $S_1, \ldots, S_g$ with increasing slopes $\l_1 < \cdots <\l_g$. We shall write $\mathcal{N}_{\phi}(F)=S_1+S_2+\cdots+S_g$. The sides of  $\mathcal{N}_{\phi}(F)$ that have negative slopes form the $\ph$-principal Newton polygon of $F(x)$,   denoted by $\mathcal{N}_{\phi}^{-}(F)$.

Let  $l(S)$ and $h(S)$ be respectively the length and the height of a side $S$ of $\mathcal{N}_{\phi}^{-}(F)$.  Then  the slope $\lambda_{S}$ of $S$ equals $-\frac{h(S)}{l(S)}=-\frac{h}{e}$ where $h$ and $e$ are two coprime  positive integers.
The degree of $S$ is $d(S)=\gcd (l(S), h(S))=\frac{d(S)}{e}$. The natural integer $e = \frac{l(S)}{d(S)}$ is called the ramification index of  $S$ and denoted by $e(S)$. 
  Define  the finite   $\ph$-residual field  $\mathbb{F}_{\phi}:=\mathbb{Z}[x]\textfractionsolidus(p,\phi (x)) \simeq \mathbb{F}_p[x]\textfractionsolidus (\overline{\phi (x)}) $. Then we attach to $S$ (or to $\lambda_{S}$) the following residual polynomial: $$R_{\lambda_{S}}(F)(y)=t_{s+de}y^d+t_{s+(d-1)e}y^{d-1}+\cdots+t_{s+e}y+c_s \in \mathbb{F}_{\phi}[y],$$ where  $$t_{s+ke} = \left \{\begin{array}{ll}
	\dfrac{a_{s+ke}(x)}{p^{v_p (a_{s+ke}(x))}} \mod{(p,\phi(x))},& \emph{if } \big(s+ke,v_p(a_{s+ke}(x))\big) \mbox{ lies on  }S,\\
	\\
	0,&  \emph{if} \big(s+ke, v_p(a_{s+ke}(x))\big) \mbox{ lies strictly above }S.\\
\end{array}
\right.$$

 Let    
$\mathcal{N}^{-}_{\ph_i}(F)=S_{i1}+S_{i2}+\cdots+S_{ir_i}$, $e_{ij}$ is the ramification index
of  $S_{ij}$,  $\lambda_{ij}$ is the slope of $S_{ij}$ 
and  $R_{\lambda_{ij}}(F)(y) = \prod\limits_{s=1}^{s_{ij}} \psi_{ijs}(y)^{n_{ijs}}$ is the factorization of $R_{\lambda_{ij}}(F)(y)$ into a product of powers of distinct irreducible polynomials in $\mathbb{F}_{\phi_i}[y], i = 1,\ldots,t$ and $j = 1,\ldots,r_i$. We say that  $F(x)$ is $\ph_i$-regular if $n_{ijs}=1$ for all $(i, j, s)$. Also,  $F(x)$ is called  $p$-regular if it is $\ph_i$-regular for all $i=1, \ldots, t$.  The following   theorem, due to Ore \cite{O}, will  play a key role  for proving   our results (see \cite[Theorems 1.13, 1.15 and 1.19]{Nar}):
\begin{theorem}[Theorem of Ore]\label{ore} \
	\\With the above notations, if  $F(x)$ is $p$-regular, then  
	$$p\mathcal{O}_K=\prod_{i=1}^t\prod_{j=1}^{r_i}
	\prod_{s=1}^{s_{ij}}\pF^{e_{ij}}_{ijs},$$ where $\pF_{ijs}$ are distinct  prime ideals of $\mathcal{O}_K$ with  $f(\pF_{ijs}/p)=\deg(\ph_i)\times \deg(\psi_{ijs})$   for every $(i, j, s)$.
	\end{theorem} 
	\section{Proofs of the main results}
	In this section, we prove our main results. Recall that  $K=\Q(\a)$ with $F(\a)=0$, where $F(x)=x^4+ax^3+bx+c$, an irreducible polynomial in  $\Q(x)$. Let us start by proving Theorem \ref{p=2}. 
\begin{proof}[\textbf{Proof of Theorem \ref{p=2}}] In all cases, we  show that $2$ is a common index divisor of $K$ and we determine $\n_2(i(K))$. 

 In Cases (1) and (2), we have   
	  $(a, b, c) \equiv (1, 1, 1) \md{2}$. Reducing modulo $2$, we see that    $F(x) \equiv \ph_1(x)^2 \ph_2(x) \md{2}$, where $\ph_1(x)=x-1$ and $\ph_2(x)=x^2+x+1$. By Theorem \ref{ore},  the factor $\ph_2(x)$ provides a unique prime ideal $\pF_{211}$ of $\mathcal{O}_K$ with residue degree $2$   above $2$. By Table \ref{Tablequartic}, $2$ divides  $i(K)$ if and only if  the factor $\ph_1(x)$ provides an other prime  ideal of  $\mathcal{O}_K$ with residue degree $2$  above $2$. The $\ph_1$-adic development of $F(x)$ is 
	\begin{equation}\label{Dev1}
F(x)=1+a+b+c+(4+3a+b)\ph_1(x)+(6+3a)\ph_1(x)^2+(4+a)\ph_1(x)^3+\ph_1(x)^4.
	\end{equation}
Let $\mu_0=\n_2(1+a+b+c)$ and $ \mu_1=\n_2(4+3a+b)$. Thus, $\mathcal{N}_{\phi_1}^{-}(F)$ is the  lower convex hull  of points $(0, \mu_0), (1, \mu_1)$ and $(2, 0)$. \\
		$(1)$ $(a, b, c) \in \{ (1, 3, 7), (1, 7, 3), (3, 1, 7), (3, 5, 3),  (5, 3, 3), (5, 7, 7), (7, 1, 3), (7, 5, 7)\} \md{8}$. Here, we have $\mu_0=2$ and $\mu_1=1$. By (\ref{Dev1}), $\mathcal{N}_{\phi_1}^{-}(F)=S_{11}$  has a single side of degree $2$ joining the points $(0, 2), (1, 1)$ and $(2, 0)$ (see FIGURE 1). Its  ramification index equals  $1$  and its attached residual polynomial is  $R_{\l_{11}}(F)(y)=1+y+y^2 \in \F_{\ph_1}[y]$. Note that $\F_{\ph_1} \simeq \F_{2}$, because $\mbox{deg}(\ph_1(x))=1$. Thus, $R_{\l_{11}}(F)(y)$ is separable. So, $F(x)$ is $\ph_1$-regular. Hence,  it is regular. Applying Theorem \ref{ore}, we see that $$2\mathcal{O}_K=\pF_{111} \cdot  \pF_{211},$$ where $\pF_{111}$ and  $\pF_{211}$ are two distinct prime ideals of  $\mathcal{O}_K$ of  residue degree $2$ each. That is $2\mathcal{O}_K=[2^1, 2^1]$. By Table \ref{Tablequartic}, $2$ is a common index divisor of $K$. Furthermore, we have $\n_2(i(K))=1$.
			\begin{figure}[htbp] 
			\centering
			\begin{tikzpicture}[x=2cm,y=1cm]
				\draw[thick] (-0.2,0) -- (2.2,0);
				\draw[thick] (0,-0.2) -- (0,1.5);
				\draw (0,0.5) node {-};
				\draw (0,0.5) node[left]{$1$};
				\draw (0,1) node {-};
				\draw (0,1) node[left]{$2$};
				\draw (1,0) node {$\shortmid$};
				\draw (1,0) node[below] {$1$};
				\draw (2,0) node {$\shortmid$};
				\draw (2,0) node[below] {$2$};
				\draw (2,0) node {$\bullet$};
				\draw (0,1) node {$\bullet$};
				\draw(2,0)--(0,1);
				\draw (1,0.5)node{$\bullet$};	
			\end{tikzpicture}
			\caption{\small  $\mathcal{N}_{\phi_1}^{-}(F)$ in Case (1)}
		\end{figure}
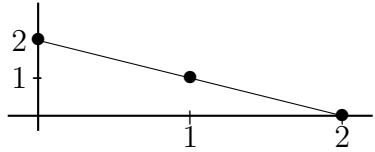\\
	$(2)$ 
To be concise, consult  the  theorem statement  for the list of triplets $(a, b, c)$ in this case. 	Here, we have $\mu_0=4$ and $\mu_1=2$. This case is similar to the above one. 
	Therefore,  $2\mathcal{O}_K=[2^1,  2^1]$. Consequently, $\n_2(i(K))=1$.  

 In Cases (3)-(8), we have $(a, b, c) \equiv (0, 0, 1) \md{2}$. Here, we have  $F(x) \equiv (x-1)^4 \md{2}$. Let $\ph_1(x)=x-1$ and  keeping in mind the $\ph_1$-adic development (\ref{Dev1}) of $F(x)$. Note that in this case the length of   $\mathcal{N}_{\phi_1}^{-}(F)$ equals $4$, because  $\n_{\ol{\ph_1(x)}} (\ol{F(x)})=4$. Let  $\mu_2=\n_2(6+3a)$ and $\mu_3= \n_2(4+a)$. Thus, by (\ref{Dev1}), $\mathcal{N}_{\phi_1}^{-}(F)$ is the Newton polygon determined by   the points  $(0, \mu_0), (1, \mu_1), (2, \mu_2), (3, \mu_3)$ and $(4, 0)$. We discuss each case separately.\\
$(3)$ $(a, b, c) \in \{(0, 0, 15), (0, 8, 7), (4, 4, 7), (4, 12, 15), (8, 0, 7), (8, 8, 15), (12, 4, 15),  (12, 12, 7)\} \md{16}$. Here, we have $\mu_0 \ge 4, \mu_1= 2$ and $\mu_2=1$. By (\ref{Dev1}), $\mathcal{N}_{\phi_1}^{-}(F)=S_{11}+S_{12}+S_{13}$ has three distinct  sides of degree $1$ each joining the points $(0, \mu_0), (1, 2), (2, 1)$ and $(4, 0)$ with respective ramification indices $e_{11}=e_{12}=1$ and $e_{13}=2$ (see FIGURE 2).   Thus, $R_{\l_{1k}}(F)(y)=1+y \in \F_{\ph_1}[y], k=1, 2, 3$ which are separable as they are of degree $1$ each.  Using Theorem \ref{ore}, we see that  $2\mathcal{O}_K=[1^1, 1^1, 1^2]$. So, by Table \ref{Tablequartic},  $2$ is a common index divisor of $K$. Furthermore, $\n_2(i(K))=1$.
	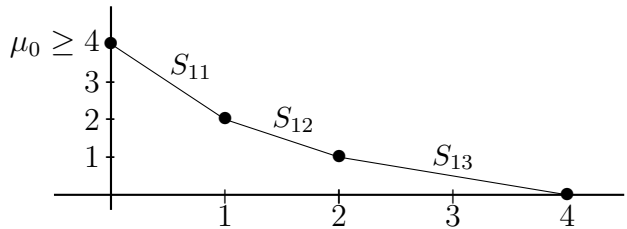
\begin{figure}[htbp] 
	\centering
	\begin{tikzpicture}[x=1.5cm,y=1cm]
		\draw[thick] (-0.5,0) -- (4.5,0);
		\draw[thick] (0,-0.2) -- (0,2.5);
		\draw (0,0.5) node {-};
		\draw (0,0.5) node[left]{$1$};
		\draw (0,1) node {-};
		\draw (0,1) node[left]{$2$};
		\draw (0,1.5) node {-};
		\draw (0,1.5) node[left]{$3$};
		\draw (0,2) node {-};
		\draw (0,2) node[left]{$ \mu_0 \ge 4$};
		\draw (1,0) node {$\shortmid$};
		\draw (1,0) node[below] {$1$};
		\draw (0.7,1.7) node { \small $S_{11}$};
	\draw (1.6,1) node { \small $S_{12}$};
		\draw (3,0.5) node { \small $S_{13}$};
		\draw (2,0) node {$\shortmid$};
		\draw (2,0) node[below] {$2$};
			\draw (2,0) node {$\shortmid$};
		\draw (3,0) node[below] {$3$};
			\draw (3,0) node {$\shortmid$};
		\draw (4,0) node[below] {$4$};
		\draw (4,0) node {$\bullet$};
		\draw (0,2) node {$\bullet$};
		\draw (1,1) node {$\bullet$};
		\draw (2,0.5) node {$\bullet$};
		\draw(4,0)--(2,0.5)--(1,1)--(0,2);	
	\end{tikzpicture}
	\caption{ \small  $\mathcal{N}_{\phi_1}^{-}(F)$ in Case (3)}
\end{figure}\\
 $(4)$  $(a, b) \in \{(0, 4), (4, 8), (8, 12), (12, 0)\} \md{16}$ and $ a+b \equiv -(1+c) \md{64}$.  In this case, we have $\mu_0 \ge 6, \mu_1= 3$ and $\mu_2=1$. By (\ref{Dev1}), $\mathcal{N}_{\phi_1}^{-}(F)=S_{11}+S_{12}+S_{13}$ has three distinct  sides of degree $1$ each joining the points $(0, \mu_0), (1, 3), (2, 1)$ and $(4, 0)$. Similarly to the above case, $\n_2(i(K))=1$.\\
 $(5)$ $(a, b) \in \{(2, 6), (10, 14), (18, 22), (26, 30)\} \md{32}$ and $a+b \equiv -c+63 \md{128}$. Here, we have $\mu_0=6, \mu_1=4, \mu_2=2$ and $\mu_3=1$. According to  (\ref{Dev1}),  $\mathcal{N}_{\phi_1}^{-}(F)=S_{11}+S_{12}$ has two distinct  sides of degree $2$ each joining the points  $(0, 6),  (2, 2)$ and $(4, 0)$ (see FIGURE 3). Further, we have $R_{\l_{11}}(F)(y)= R_{\l_{12}}(F)(y)=1+y+y^2$ which are  separable in  $\F_{\ph_1}[y]$. Thus, $F(x)$ is $\ph_1$-regular, and so it is $2$ regular. By Theorem \ref{ore},  $2\mathcal{O}_K=[2^1, 2^1]$. Hence, by Table  \ref{Tablequartic}, $2$ divides  $i(K)$ and $\n_2(i(K))=1$.\\
 	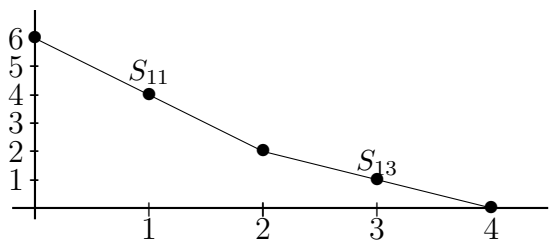
\begin{figure}[htbp] 
 	\centering
 	\begin{tikzpicture}[x=1.5cm,y=0.75cm]
 		\draw[thick] (-0.2,0) -- (4.5,0);
 		\draw[thick] (0,-0.2) -- (0,3.5);
 		\draw (0,0.5) node {-};
 		\draw (0,0.5) node[left]{$1$};
 		\draw (0,1) node {-};
 		\draw (0,1) node[left]{$2$};
 		\draw (0,1.5) node {-};
 		\draw (0,1.5) node[left]{$3$};
 		\draw (0,2) node {-};
 		\draw (0,2) node[left]{$  4$};
 		\draw (0,2.5) node[left]{$5$};
 			\draw (0,2.5) node {-};
 			\draw (0,3) node[left]{$6$};
 		\draw (1,0) node {$\shortmid$};
 		\draw (1,0) node[below] {$1$};
 		\draw (1,2.4) node { \small $S_{11}$};
 		\draw (3,0.8) node { \small $S_{13}$};
 		\draw (2,0) node {$\shortmid$};
 		\draw (2,0) node[below] {$2$};
 		\draw (2,0) node {$\shortmid$};
 		\draw (3,0) node[below] {$3$};
 		\draw (3,0) node {$\shortmid$};
 		\draw (4,0) node[below] {$4$};
 		\draw (4,0) node {$\bullet$};
 		\draw (0,3) node {$\bullet$};
 		\draw (1,2) node {$\bullet$};
 		\draw (2,1) node {$\bullet$};
 		\draw (3,0.5) node {$\bullet$};
 		\draw(4,0)--(3,0.5)--(2,1)--(1,2)--(0,3);	
 	\end{tikzpicture}
 	\caption{\small  $\mathcal{N}_{\phi_1}^{-}(F)$ in Case (5)}
 \end{figure}\
 \\$(6)$  $(a, b) \in \{(6, 42), (22, 122), (38, 74), (54, 26), (70, 106), (86, 58), (102, 10), (118, 90)\} \md{128}
 $ and $a+b \equiv -(1+c) \md{1024}$. Here, we have $\mu_0 \ge 10, \mu_1=6, \mu_2=3$ and $\mu_1=1$. Thus, $\mathcal{N}_{\phi_1}^{-}(F)=S_{11}+S_{12}+S_{13}+S_{14}$ has four distinct sides of degree $1$ each. By Theorem \ref{ore},  $2\mathcal{O}_K=[1^1, 1^1, 1^1, 1^1]$. Therefore,  by Table \ref{Tablequartic}, $\n_2(i(K))=2$.\\
$(7)$ $(a, b, c) \equiv (6, 10, 15) \md{16},  \mu > 7$ is even and $2 \n > \mu+3$. In this case,  $\mu_0= \mu, \mu_1= \n, \mu_2=3$ and $\mu_1=1$.  It follows that  $\mathcal{N}_{\phi_1}^{-}(F)=S_{11}+S_{12}+S_{13}$ has three distinct  sides of degree $1$ each joining the points $(0, \mu),  (2, 3), (3, 1)$ and $(4, 0)$ with respective ramification indices $e_{11}=2, e_{12}= e_{13}=1$. Using  Theorem \ref{ore}, we see that  $2\mathcal{O}_K=[1^2, 1^1, 1^1]$. Hence, by Table \ref{Tablequartic},   $\n_2(i(K))=1$.\\
$(8)$ $(a, b, c) \equiv (6, 10, 15) \md{16}, \n>5$ and $ 2 \n < \mu+3$. Here,  $\mathcal{N}_{\phi_1}^{-}(F)=S_{11}+S_{12}+S_{13}+S_{14}$ has four distinct  sides of degree $1$ each joining the points $(0, \mu),  (1, \n), (2, 3), (3, 1)$ and $(4, 0)$.  By Theorem \ref{ore},  $2\mathcal{O}_K=[1^1, 1^1,  1^1, 1^1]$. According to  Table \ref{Tablequartic},  we have  $\n_2(i(K))=2$.

 From Case (9), we have  $(a, b, c) \equiv (1, 0, 0) \md{2}$. Thus, $F(x) \equiv \ph_1(x)^3 \ph_2(x) \md{2}$, where $\ph_1(x)=x$ and $\ph_2(x)=x-1$. By Theorem \ref{ore}, the factor $\ph_2(x)$ provides a unique prime ideal of $\mathcal{O}_K$ of residue degree $1$  above $2$. It follows  by Table \ref{Tablequartic} that   $2$ divides $i(K)$ if and only if the factor $\ph_1(x)$ provides  at least two distinct  prime ideals of residue degree $1$ each lying above $2$. Since $F(x)= \ph_1(x)^4+a \ph_1(x)^3+b \ph_1(x)+c$, we have  $\mathcal{N}_{\phi_1}^{-}(F)$ is the Newton polygon determined by  the points $(0, \n_2(c)), (1, \n_2(b))$ and $(3, 0)$.\\
$(9)$ $3 \n_2(b)<2\n_2(c)$ and $\n_2(b)$ is odd. In this case, $\mathcal{N}_{\phi_1}^{-}(F)=S_{11}+S_{12}$ has two distinct sides of degree $1$ each joining the points $(0, \n_2(c)), (1, \n_2(b))$ and $(3, 0)$ with $e_{11}=1$ and $e_{12}=2$. Applying Theorem \ref{ore}, we get that  $2 \mathcal{O}_K =[1^1, 1^1, 1^2]$. 	Hence, by Table \ref{Tablequartic}, $2$ is a common index divisor of $K$ and $\n_2(i(K))=1$.

 In the rest of the  cases  $3 \n_2(b)<2\n_2(c)$ and $\n_2(b)$ is even.  It follows that,  $\mathcal{N}_{\phi_1}^{-}(F)=S_{11}+S_{12}$ has two distinct sides with $d(S_{11})=1$ and $d(S_{12})=2$. Further,  $R_{\l_{12}}(F)(y)=(y-1)^2 \in \F_{\ph_1}[y]$ which is not separable. Thus, $F(x)$ is not $\ph_1$-regular.  Set $b = 2^{2k}b_2$ for some positive integer $k$. In order to apply Theorem \ref{ore}, let us replace $\ph_1(x)$ by $\psi_1(x)=x-s$ with $s=2^kb_2$. The $\psi_1$-adic development of $F(x)$ is given by 
\begin{eqnarray}\label{Devpsi1}
	F(x)&=&s^4+as^3+bs+c+ (4s^3+3as^2+b) \psi_1(x)+ 3s(2s+a)\psi_1(x)^2 \nonumber \\
	&+&(4s+a)\psi_1(x)^3+\psi_1(x)^4.
\end{eqnarray}
Set $\omega_0= \n_2(s^4+as^3+bs+c)$ and $ \omega_1= \n_2(4s^3+3as^2+b)$ and remark that $\n_2(3s(2s+a))=k$.  Thus, by (\ref{Devpsi1}), $\mathcal{N}_{\psi_1}^{-}(F)$ is the lower convex hull of the points $(0, \omega_0), (1, \omega_1), (2, k)$ and $(3, 0)$.   Since $3 \n_2(b)<2\n_2(c)$,  $\n_2(c)>3k$. Set   $\n_2(c) =3k+l$ with  $l$ is a positive integer.   It follows that  $\omega_0 = \n_2(2^{4k}b_2^4+2^{3k}b_2^2(ab_2+1)+2^{3k+l}c_2)$ and $ \omega_1 = \n_2(2^{3k+2}b_2^3+2^{2k}b_2(3ab_2+1))$. 

Let us  study each case separately.
\\
$(10)$ $ b \equiv 4 \md{8}, ab_2+1 \equiv 0 \md{4}$ and $c \equiv 0 \md{32}$.	In this case, we have $\omega_0=4$ and $\omega_1=3$. Thus,  $\mathcal{N}_{\psi_1}^{-}(F)=S_{11}+S_{12}$ has two distinct  sides of degree $1$ each joining the points $(0, 4), (2, 1)$ and $(3, 0)$.  Their respective ramification indices are   $e_{11}=2$ and $e_{12}=1$. Using Theorem \ref{ore}, we get $2\mathcal{O}_K=[1^1, 1^2, 1^1]$. 		In view of Table \ref{Tablequartic},  $2$ is a common index divisor of $K$ and $\n_2(i(K))=1$. \\
$(11)$ $b \equiv 4 \md{8}, ab_2+1 \equiv 2 \md{4}$ and $c \equiv 16 \md{32}$. Here also we have, $\omega_0=4$ and $\omega_1=3$. As in the above case,  $\n_2(i(K))=1$. \\
$(12)$  $b \equiv 4 \md{8}, ab_2+1 \equiv 0 \md{8},  c \equiv 16  \md{32}$ and $b_2^4+c_2  \equiv 0 \md{4}$. In  this case, we have $\omega_0  \ge 6$ and  $\omega_1=3 $. It follows that   $\mathcal{N}_{\psi_1}^{-}(F)=S_{11}+S_{12}+S_{13}$ has three distinct  sides of degree $1$ each and   having   ramification index $1$ each. Precisely, $\mathcal{N}_{\psi_1}^{-}(F)$ is  the Newton polygon joining the points $(0, \omega_0), (1, 3), (2, 1)$ and $(3, 0)$. Using Theorem \ref{ore}, we see that  $2\mathcal{O}_K=[1^1, 1^1, 1^1, 1^1]$.  Consequently, $2$ is a common index divisor of $K$ and $\n_2(i(K))=2$.  \\
$(13)$ $ \n_2(b)=2k, ab_2+1 \equiv 0 \md{4}$ and  $\n_2(c)=3k+1$ with $k \ge 2$.   Here,  $\omega_0=3k+1$ and   $\omega_1=2k+1$. Thus,  $\mathcal{N}_{\phi_1}^{-}(F)=S_{11}+S_{12}$ has two distinct sides of degree $1$ each  joining the points $(0, 3k+1), (2, k)$ and $(3, 0)$. Their ramification indices are   $e_{11}=2$ and $e_{12}=1$.  Using Theorem \ref{ore}, we see that $2\mathcal{O}_K=[1^2, 1^1, 1^1]$. Therefore, by Table \ref{Tablequartic},  $2$ is a common index divisor of $K$ and $\n_2(i(K))=1$.  \\
$(14)$ $\n_2(b)=2k,   ab_2+1 \equiv 2 \md{4}$ and $\n_2(c)=3k+l$ with $k, l \ge 2$.   Here, we have $\omega_0=3k+1$ and $\omega_1 \ge 2k+2$. It follows that  $\mathcal{N}_{\phi_1}^{-}(F)$ is the same as in the above case. Therefore, $\n_2(i(K))=1$. \\
$(15)$   $\n_2(b)=2k, ab_2+1 \equiv 0 \md{4}, \n_2(c)=3k+l$ and $2^{k-2}b_2^4+b_2^2 \cdot \frac{ab_2+1}{4}+2^{l-2}c_2 \equiv 0 \md{2}$ with $k \ge 2$ and $l \ge 2$. In this case, we have $\omega_0= 3k+2+ \n_2\big(2^{k-2}b_2^4+b_2^2 \cdot \frac{ab_2+1}{4}+2^{l-2}c_2\big) \ge 3k+3$ and $ \omega_1=2k+1$.   It follows that  $\mathcal{N}_{\phi_1}^{-}(F)=S_{11}+S_{12}+S_{13}$ has three  distinct sides of degree $1$ each joining the points $(0, \omega_0), (1, 2k+1), (2, k)$ and $(3, 0)$. Using Theorem \ref{ore}, we see that   $2\mathcal{O}_K=[1^1, 1^1, 1^1, 1^1]$.  By Table \ref{Tablequartic}, $2$ is a common index divisor of $K$ and $\n_2(i(K))=2$. 

$\bullet$ To treat Cases (16) and (17), we shall use the following two parameters:   $\t =  \n_2\big(b_2^4+b_2 \frac{ab_2+1}{2}+2^{l-1}c_2\big)$ and $\sigma = \n_2\big(2b_2^2+\frac{3ab_2+1}{4}\big)$. \\
$(16)$ $b \equiv 4 \md{8},  ab_2+1 \equiv 2 \md{4}, \n_2(c)=2+l$ with $l \ge 2$, $\t < 3+2 \sigma $ and  $\t$ is even.  Here, we have $\omega_0= 4+ \t$ and $\omega_1= 4+ \sigma$. Also, we have  $\mathcal{N}_{\phi_1}^{-}(F)=S_{11}+S_{12}$ has two distinct sides of degree $1$  each joining the points $(0, \omega_0), (2, 1)$ and $(3, 0)$. Applying Theorem \ref{ore} and using Table \ref{Tablequartic}, we see that  $\n_2(i(K))=1$. \\ 
$(17)$ $b \equiv 4 \md{8},  ab_2+1 \equiv 2 \md{4}, \n_2(c)=2+l$ with $l \ge 2$ and $\t > 3+2 \sigma $. In this case,   $\mathcal{N}_{\phi_1}^{-}(F)=S_{11}+S_{12}+S_{13}$ has three distinct sides  of degree $1$ each. Therefore, by Theorem \ref{ore} and Table \ref{Tablequartic},  $\n_2(i(K))=2$.  This completes the proof of the theorem.
\end{proof}
Now, let us prove Theorem \ref{p=3}.
\begin{proof}[\textbf{Proof of Theorem \ref{p=3}}] In every  case, we show  that $i(K)$ is divisible by $3$. According to  Table \ref{Tablequartic}, $3$ is a common index divisor of $K$ if and only if  there exist four distinct  prime ideals of  $\mathcal{O}_K$ with residue degree $1$ each dividing $3$. Equivalently,  $3\mathcal{O}_K=[1^1, 1^1, 1^1, 1^1]$. Furthermore, $\n_3(i(K))=1$ whenever $3$ divides $i(K)$. \\
	$(1)$ $a \equiv \pm 1 \md{3}$, $b \equiv c \equiv 0 \md{3}$, $ \, 3  \n_3(b) < 2 \n_3(c), \, \n_3(b)$ is even and $b_3 \equiv -a \md{3}$.  In this case, $F(x) \equiv \ph_1(x)^3 \ph_2(x) \md{3}$,  where $\ph_1(x)=x$ and $\ph_2(x)=x + a$. By Theorem \ref{ore}, the factor $\ph_2(x)$ provides a unique prime ideal of $\mathcal{O}_K$ dividing  $3$ with residue degree $1$. On the other hand, we have  $\mathcal{N}_{\phi_1}^{-}(F)=S_{11}+S_{12}$ has two distinct  sides joining the points $(0, \n_3(c)), (1, \n_3(b))$ and $(3, 0)$. Their degrees are   $d(S_{11})=1$ and $d(S_{12})=2$. Further,  $R_{\l_{11}}(F)(y)$ is separable as it is of degree $1$, and $R_{\l_{12}}(F)(y)= b_3  +ay^2 = a(y^2+ b_3a^{-1}) = a(y^2-1)=a (y-1)(y+1)$ which  is separable  in $\F_{\ph_1}[y]$. So, $F(x)$ is $\ph_1$-regular. Therefore, by Theorem \ref{ore}, $3 \mathcal{O}_K=[1^1, 1^1, 1^1, 1^1]$. Consequently,  $3$ is a common index divisor of $K$.\\
$(2)$ $(a, b) \in \{(0, 4), (3, 22), (6, 13), (9, 4), (12, 22), (15, 13), (18, 4), (21, 22), (24, 13)\} \md{27},  a+b  \equiv 1+c \md{243}$. Here, we have $F(x) \equiv \ph_1(x)^3 \ph_2(x) \md{3}$,  where $\ph_1=x+1$ and $\ph_2(x)=x$. By Theorem \ref{ore}, the factor $\ph_2(x)$ provides a unique prime ideal  of $\mathcal{O}_K$ of residue degree  $1$ above $3$. It follows that  $3$ divides $i(K)$ if and only the factor $\ph_1(x)$  provides three distinct prime ideals of  $\mathcal{O}_K$ with residue degree  $1$ each lying above $3$. The $\ph_1$-adic development of $F(x)$ is 
\begin{eqnarray}\label{DEVTHREE1}
F(x)= 1-a-b+c+ (-4+3a+b) \ph_1(x)- 3(a-2)\ph_1(x)^2+(a-4)\ph_1(x)^3+\ph_1(x)^4.
\end{eqnarray}
Let   $  \m= \n_3(1-a-b+c), \n= \n_3(-4+3a+b)$. Remark that  $\n_3\big(-3(a-2)\big)=1$.   It follows that $\mathcal{N}_{\phi_1}^{-}(F)$ is the Newton polygon determined by  the points  $(0, \mu), (1, \n), (2, 1)$ and $(3, 0)$. Note also that $\mu \ge 5$ and $\n \ge 3$.  
	\begin{enumerate}[label=(\roman*)]
		\item If  $2 \n < \mu +1 $, then $\mathcal{N}_{\phi_1}^{-}(F)=S_{11}+S_{12}+S_{13}$ has three distinct sides  of degree $1$ each. Therefore,  $3\mathcal{O}_K=[1^1, 1^1, 1^1, 1^1]$. Hence, $3 $ divides  $ i(K)$.
		\item If $2 \n >  \mu + 1$,  $\mu$  is odd and $\big(1-a-b+c\big)_3 \equiv 1 \md{3}$, then $\mathcal{N}_{\phi_1}^{-}(F)=S_{11}+S_{12}$ has two distinct sides with $d(S_{11})=2$ and $d(S_{12})=1$. Since $\big(1-a-b+c\big)_3 \equiv 1 \md{3}$, $R_{\l_{11}}(F)(y)= \big(1-a-b+c\big)_3 - (a-2) y^2= 1-y^2 = (1-y)(1+y) \in \F_{\ph_1}[y]$. Thus, $F(x)$ is $\ph_1$-regular. Therefore, by Theorem \ref{ore},   $3\mathcal{O}_K=[1^1, 1^1, 1^1, 1^1]$. Consequently, $3$ divides  $i(K)$.
\end{enumerate}
 $(3)$  $(a, b) \in \{(0, 23), (3, 14), (6, 5), (9, 23), (12, 14), (15, 5), (18, 23), (21, 14), (24, 5)\} \md{243},  a+b \equiv -(1+c) \md{243}$. In this case,  $F(x) \equiv  \ph_1(x)^3 \ph_2(x) \md{3}$, where $\ph_1=x-1$ and $\ph_2(x)=x$.   The $\phi_1$-adic development of $F(x)$ is 
 \begin{equation}\label{Devphi10-113}
 F(x)=1+a+b+c+(4+b+3a) \phi_1(x)+ 3(2+a)\phi_1(x)^2+(4+a)\phi_1(x)^3+\phi_1(x)^4.
 \end{equation}
 Let $\omega= \n_3(1+a+b+c)$ and $\t= \n_3(4+3a+b)$.   Since $\n_3(3(2+a))=1$,  by (\ref{Devphi10-113}),  $\mathcal{N}_{\phi_1}^{-}(F)$ is the lower convex hull of the points $(0, \omega), (1, \t), (2, 1)$ and $(3, 0)$. We proceed as in the above case, we obtain the desired result as stated in the theorem. \\	
 $(4)$  $(a, b) \in \{ (2, 16), (5, 7),  (8, 25),  (11, 16), (14, 7), (17, 25),   (20, 16), (23, 7),  (26, 25)\} \md{27}\}$ and $a+b \equiv 1+c \md{81}$. Here, we have  $(a, b, c) \equiv (-1, 1, -1) \md{3}$.  Thus, $F(x)= \ph_1(x)^3 \ph_2(x) \md{3}$, where $\ph_1(x)=x+1$ and $\ph_2(x)=x-1$.  Recall the $\ph_1$-adic development (\ref{DEVTHREE1}) of $F(x)$ and  $\mu$ and $\n$ as  given in Case (2).  Here, we have $\mu  \ge 4$ and $  \n = 2$.  Further, we have  $\n_3(-3(a-2)) \ge 2$. It follows that 
 $\mathcal{N}_{\phi_1}^{-}(F)=S_{11}+S_{12}$ has two distinct sides joining  the points $(0, \mu), (1, 2)$ and $(3, 0)$. Thus, $d(S_{11})=1$ and $d(S_{12})=2$. Moreover, we have $R_{\l_{12}}(F)(y)= \big(-4+3a+b\big)_3+(a-4)y^2=-1+y^2=(y-1)(y+1) \in \F_{\ph_1}[y]$. So, $F(x)$ is $\ph_1$-regular. Therefore, by Theorem \ref{ore}, $3\mathcal{O}_K=[1^1, 1^1, 1^1, 1^1]$. Hence, $3$ divides $i(K)$.  This ends the proof of the theorem.
\end{proof}	
\begin{proof}[\textbf{Proof of Theorem \ref{monogenic}}] Since $p$ divides both $a, b$ and $c$,    $F(x) \equiv \ph_1(x)^4  \md{p}$, where $\ph_1(x)=x$.  As $\min\{ 4\n_p(a), \frac{4}{3} \n_p(b)\} > \n_p(c)  > 2$,     $\mathcal{N}_{\ph_1}^{-}(F)=S_{11}$ has a single side  joining the points $(0, \n_p(c))$ and $(4, 0)$.    Moreover, the degree of $S_{11}$ equals $1$, because $\n_p(c)$ is odd. Thus, $R_{\l_{11}}(F)(y) $ is separable in $\F_{\ph_1}[y]$.  By the theorem of the residual polynomial \cite{Nar}, $F(x)$ is irreducible in  $\Q_p(x)$. Consequently, $F(x)$ is irreducible in $\Q(x)$.	Let $K=\Q(\a)$ with $\a$ a root of $F(x)$. Using Theorem 4.18 of \cite{Nar}, we see that $$\n_p(\mbox{ind}(\a))= \dfrac{3(\n_p(c)-1)}{2} \ge 3.$$ Thus, $p$ divides $\mbox{ind}(\a)$. So, $ \Z[\a] \neq \mathcal{O}_K$, that is $F(x)$ is not monogenic. Let $(s, t)$ be the unique positive solution of the Diophantine equation $ \n_p(c) x - 4y=1$ with $0 \le s < 4$. Let $\th=\frac{\a^s}{p^t}$. One can easily  check that $K=\Q(\th)$.  Let us show that $\mathcal{O}_K=\Z[\th]$. Let  $K_p(\th) := \Q_p(\th)$.  Since $\Q_p$ is a Henselian field, there is a unique valuation $\omega_p$ extending $\n_p$ to $K_p(\th)$.  On the other hand, by \cite[Corollary 3.1.4]{Endler},    $\mathcal{O}_K=R_{\omega_p}$, where $R_{\omega_p}$ is the valuation ring of $\omega_p$.  Therefore, if  $\mathcal{O}_K=\Z[\th]$, then it is necessary  that   $\omega_p (\th) \ge 0$.   So, let us  first  prove   that  $\omega_p (\th) \ge 0$.  Since  $\mathcal{N}_{\ph_1}^{-}(F)=S_{11}$ is only one side with slope $\l_{11}= \frac{-\n_p(c)}{4}$, $\omega_p(\a)= \frac{\n_p(c)}{4}$. It follows that $$\omega_p(\th)= \omega_p \big( \frac{\a^s}{p^t}\big)=s \omega_p(\a)-t= s \times \frac{\n_p(c)}{4}-t= \frac{s \n_p(c)-4t}{4}=\frac{1}{4} \ge 0.$$ Let $M(x)$ be the minimal polynomial of $\th$  over $\Q$ and $ \th_1=\th, \th_2, \th_3, \th_3$ be the $\Q_p$-conjugate of $\th$.  By using the well-known relation linking  coefficients and  roots of a monic polynomial, one has:
$$M(x)=x^4- \sigma_1(\th_1, \th_2, \th_3, \th_4)x^3+ \sigma_2(\th_1, \th_2, \th_3, \th_4)x^2-\sigma_3(\th_1, \th_2, \th_3, \th_4)x+ \sigma_4(\th_1, \th_2, \th_3, \th_4),$$
where $\sigma_i (x_1, x_2, x_3,x_4), i=1, 2, 3, 4$ are the elementary symmetric polynomials in $4$ variables. 
  Note that $\omega_p(\th_i) = \omega_p(\th)$ for all $i =1, 2, 3, 4$. It follows that $ \n_p(\sigma_i(\th_1, \th_2, \th_3, \th_4))= \omega_p(\sigma_i(\th_1, \th_2, \th_3, \th_4)) \ge \frac{i}{4}$ for $i=1, 2, 3$ and $ \n_p(\sigma_4(\th_1, \th_2, \th_3, \th_4))= \omega_p(\sigma_4(\th_1, \th_2, \th_3, \th_4))=  \omega_p(\th_1\th_2\th_3\th_4)= 1$. Thus, $M(x)$ is a $p$-Eisenstein polynomial. So, $p$ does not divide $\mbox{ind}(\th)$. Since $D_p$ is square free, the only prime which can divide $\mbox{ind}(\a)$ is $p$, and by definition of $\th$, $p$ is the unique prime which can divide the index $(\Z[\a] : \Z[\th])$. Consequently, $\mathcal{O}_K = \Z[\th]$.
\end{proof}	
\begin{proof}[Proof of Corollaries \ref{mongentrinom1} and \ref{mongentrinom2}] Recall  that the discriminant of $x^4+bx+c$ and $x^4+ax^3+c$ are  $256c^3-27b^4$ and  $c^2(256c-27a^4)$  respectively.  Then , by a direct application of  Theorem \ref{monogenic} and performing some computations, the two corollaries are immediately deduced.
\end{proof}
\section{Examples}
To illustrate our results, we propose some numerical examples. Let $F(x)=x^4+ax^3+bx+c$ be an irreducible polynomial   and $K=\Q(\a)$, where  $\a$ a root of $F(x)$.
	\begin{enumerate}
		\item Let $a=4913, b= 867$  and $c= 119$. Here, $F(x)$ is a $17$-Eisenstein polynomial. Hence it is irreducible over $\Q$.   As $(a, b, c) \equiv (-1, 0, -1) \md{3}$, $F(x)$ is irreducible  modulo $3$. Therefore, $3 \mathcal{O}_K$ is a prime ideal. Hence,   $3 \nmid i(K)$.  Since $(a, b, c) \equiv (1, 3, 7) \md{8}$. Then, by Theorem \ref{p=2}, $\n_2(i(K))=1$. Consequently, $i(K)=2$, and  $K$ is not monogenic. 
		\item Let  $a=25, b=1125$ and $c=405$. In this case, $F(x)$ is an Eisenstein polynomial with respect to the prime  $5$. Hence it is irreducible in $\Q(x)$. For $p=2$, $F(x) \equiv (x-1)^2 (x^2+x+1) \md{2}$. Here, $F(x)$ is not $(x-1)$-regular with respect to $\n_2$. To apply Theorem \ref{ore}, let us use $\psi_1(x)=x-s$  with $s \equiv 3 \md{32}$.  Recall the $\psi_1$- adic development (\ref{Devpsi1}) of $F(x)$, we see that $\omega_0=3$ and $\omega_1=2$. It follows that $\mathcal{N}_{\psi_1}^{-}(F)$  has a single side of degree $1$ joigning the points $(0, 3)$ and $(2, 0)$. Applying Theorem \ref{ore}, we obtain that $2 \mathcal{O}_K=[1^2, 2^1]$. Hence, $2  \nmid i(K)$.  On the other hand, we have $a \equiv 1 \md{3}, \n_3(b)=2, \n_3(c)=4$ and $b_3 \equiv -1 \md{3}$. Therefore, by Theorem \ref{p=3}, $3 $ divides $i(K)$. Hence, $i(K)=3$. Consequently,  $K$ is not monogenic. 
		\item Let $a=6, b=42, c=975$. The polynomial $F(x)$ is irreducible in $\Q(x)$, because it is Eisenstein with respect to $3$. Reducing modulo $3$, we see that $F(x) \equiv \ph_1(x)^4 \md{3}$, where $\ph_1(x)=x$.  As $\n_3(975)$,  $\mathcal{N}_{\phi_1}^{-}(F)=S_{11}$ has a single side of degree $1$ joining the points $(0, 1)$ and $(4, 0)$.    Therefore, by Theorem  \ref{ore}, $3 \mathcal{O}_K=[1^4]$. Thus, $3 \nmid i(K)$.  According  to  Theorem \ref{p=2}, $\n_2(i(K))=2$. Consequently, $i(K)=4$. 
		\item Let  $a=21156911906816, b=448$ and $c=287$. Here, $F(x)$ is a $7$-Eisenstein polynomial. So, it is irreducible in $\Q(x)$. By Theorem \ref{p=2}, we have $\n_2(i(K))=1$ and by Theorem \ref{p=3}, we have $\n_3(i(K))=1$. Hence, $i(K)=6$.
	\end{enumerate}
\section*{Conclusion}	
In this paper, we have provided infinite parametric families defining both  monogenic and  non-monogenic quartic number fields defined by quadrinomials. Additionally,   we gave a partial answer to Problem  22 in   Narkiewicz's book \cite{Na} for these families of number fields.  The proofs of  our results can be applied  for  computing  integral bases and discrimiants for  such number fields (see \cite{Nar}). 	Despite a great number of available old and new   results focusing on  indices and monogenity of quartic number  fields, the problem remains open, encouraging the search of further new efficient algorithms. 
\section*{Declarations}
\textbf{Conflict of interest:} There is no conflict of interest related to this paper. The author has freely chosen this journal without any consideration.
\section*{Data availability} 
Not applicable.


\begin{thebibliography}{99}	
		\bibitem{ANH6IJAC}  S. Ahmad, T. Nakahara and A. Hameed, On certain pure sextic fields related to a problem of Hasse,  Int. J. Algebra  Comput. \textbf{26} (2016), 577--583.
	\bibitem{Bhargave} L. Alpöge, M. Bhargava and A. Shnidman,  A positive proportion of quartic fields are not monogenic yet have no local obstruction to being so, Math. Ann.  \textbf{388} (2023), 4037--4052.
	\bibitem{Akhtari} S. Akhtari, Quartic index form equations and monogenizations of quartic orders,  Essent. Number Theory, \textbf{1} (2022), 57--72.
		\bibitem{Nyul} T. Arn\'{o}czki and G. Nyul, On a conjecture concerning the minimal index of pure quartic fields, Publ. Math. Debrecen, \textbf{104} (2024), 471--478.
	\bibitem{BRM} H. Ben Yakkou,  On non-monogenic  number fields defined by trinomial of type $x^n+ax^m+b$,  Rocky Mt. J. Math.   \textbf{53} (2023), 685--699.
\bibitem{BBAMH8}  H. Ben Yakkou and B. Boudine, On the index of the octic number field defined by $x^8+ax+b$,   Acta Math. Hungar.  \textbf{170} (2023), 585--607. 
	\bibitem{BDB} H. Ben Yakkou and J. Didi, 	On monogenity of certain pure number fields of degrees $2^r\cdot3^k\cdot7^s$, Math.  Bohem. \textbf{149} (2024), 167--183.

\bibitem{BFpr} H. Ben Yakkou and L. El Fadil, On monogenity of certain  pure number fields defined by $x^{p^r}-m$,  Int. J. Number Theory, \textbf{17} (2021),  2235-2242. 
	\bibitem{BerEVGymultiply}	A. Bérczes, J. H. Evertse and K.Gy\H{o}ry,  Multiply monogenic orders,  Ann. Sc.
	Norm. Super. Pisa Cl. Sci. \textbf{(5), 12(2)} (2013), 467--497. 
	\bibitem{BGGy6}  Y. Bilu, I. Ga\'{a}l and K. Gy\H{o}ry, Index form equations in sextic fields: a hard computation,  Acta Arith.  \textbf{115}  (2004), 85--96.
	\bibitem{DS}  C. T. Davis and B. K.  Spearman, The index of quartic field defined by a trinomial $x^4+ax+b$,  J. Algebra  Appl.  \textbf{17(10)}  (2018). Article ID: 1850197, 18pp.
	\bibitem{R} R. Dedekind, \"Uber den Zusammenhang zwischen der Theorie der Ideale und der Theorie der h\"oheren Kongruenzen,   G\"ottingen Abhandlungen, \textbf{23} (1878), 1--23.

\bibitem{Endler}	O. Endler, Valuation Theory, Springer-Verlag, Berlin, (1972).
	\bibitem{Engstrom} H. T. Engstrom, On the common index divisors of an algebraic number field,  Trans. Amer. Math. Soc.  \textbf{32} (1930), 223--237.
	\bibitem{Evertse} J. H. Evertse, A survey on monogenic orders, Publ. Math. Debrecen, \textbf{79} (2011), 411--422. 
		\bibitem{EG} 	J. H.  Evertse and K.  Gy\H{o}ry,  Discriminant Equations in Diophantine Number Theory, Cambridge Univ. Press (2017).
\bibitem{Funakura} T. Funakura, On integral bases of pure quartic fields,   Math. J. Okayama Univ. \textbf{26} (1984), 27--41.
		\bibitem{G19} I. Ga\'al,  Diophantine Equations and Power Integral Bases, Theory and Algorithm, 2nd edn.,  Birkh\"auser, (Boston, 2019).
	\bibitem{GGy5} I. Ga\'{a}l and  K. Gy\H{o}ry, Index form equations in quintic fields,  Acta Arith.   \textbf{89} (1999), 379--396.
		\bibitem{GPP41993} I. Ga\'al, A. Peth\H{o} and M. Pohst,  Simultaneous representation of integers by a pair of ternary quadratic forms-with an application to index form equations in quartic number fields,   J. Number Theory,  \textbf{57} (1996), 90--104. 	
	\bibitem{Nar}  J. Gu\`{a}rdia, J. Montes and E. Nart, Newton polygons of higher order in algebraic number theory,  Tran. Amer.  Math. Soc.  \textbf{364} ( 2012), 361--416. 

\bibitem{Gyoryredecide} K. Gy\H{o}ry,	Sur les polyn\^{o}mes \`a coefficients entiers et de discriminant donné, III,  Publ. Math. Debrecen, \textbf{23} (1976), 141--165.

\bibitem{GyorySeminarFrensh} K. Gy\H{o}ry, Corps de nombres algébriques d'anneau d'entiers monogène, In: \textquotedblleft Séminaire Delange-Pisot-Poitou\textquotedblright, 20e année: 1978/1979. Théorie des nombres, Fasc. 2 (French), Secrétariat Math., Paris, 1980, pp. Exp. No. 26, 7.
\bibitem{Gyoryrdiscriminant} K. Gy\H{o}ry,  On discriminants and indices of integers of an algebraic number field,  J. Reine Angew. Math. \textbf{324} (1981), 114--126.
	\bibitem{Gyory1998} K. Gy\H{o}ry,  Bounds for the solutions of decomposable form equations,  Publ. Math. Debrecen,  \textbf{52} (1998), 1--31.
		\bibitem{Ibaraetal} R. Ibarra, H. Lembeck, M. Ozaslan, H. Smith and K. E.  Stange, 	Monogenic fields arising from trinomials,  Involve: a journal of Mathematics, \textbf{15} (2022), 299--317.
	\bibitem{JonesASM} L. Jones, Infinite families of non-monogenic trinomials, Acta Sci. Math. \textbf{87} (2021), 95--105.
		\bibitem{JW}  L. Jones and D. White, Monogenic trinomials with non-squarefree discriminant,  Int. J. Math.   \textbf{32} (2021). Article ID: 2150089, 21pp.
	\bibitem{LN} P. Llorente and E. Nart, Effective determination of the rational primes in a cubic field,  Proc. Amer. Math.  Soc.  \textbf{87} (1983), 579--585.
	\bibitem{Na} W. Narkiewicz, Elementary and Analytic Theory of Algebraic Numbers, 3rd edn., Springer Monographs in Mathematics  (Springer-Verlag, Berlin, 2004).
 \bibitem{Nartindex}  E. Nart, On the index of a number field,  Trans. Amer. Math. Soc.  \textbf{289} (1985), 171--183.
 \bibitem{OdjTogbeZigler} J. Odjoumani, A. Togbé   and V.  Ziegler,  On a family of biquadratic fields that do not admit a unit power integral basis,   Publ. Math. Debrecen,  \textbf{94} (2019), 1--19. 
	\bibitem{O} \O. Ore, Newtonsche Polygone in der Theorie der algebraischen K\"{o}rper,  Math. Ann.  \textbf{99} (1928), 84--117.
	\bibitem{PethoZigler} A. Peth\H{o} and  V. Ziegler,  On biquadratic fields that admit unit power integral basis,  Acta Math.  Hungar. \textbf{133}  (2011), 221--241.
	\bibitem{Smithquartic}   H. Smith, Two families of monogenic $S_4$ quartic number fields,   Acta Arith. 186 (2018), 257--271.
\end{thebibliography}
\end{document}